\documentclass[11pt, reqno]{amsart}      %fleqn in the options put all equations to the left

\usepackage{amssymb, amsmath, amsthm,stmaryrd}
\usepackage{hyperref}                                              %[backref] back references
\usepackage{amscd}   % for commutative diagrams
\usepackage{fullpage}
\usepackage[all]{xy} % for complicated commutative diagrams
\usepackage{MnSymbol,graphicx}
\usepackage{mathrsfs,enumitem}
\usepackage{nccmath} %for centering numbered formulas
\usepackage{multirow,qtree}% http://ctan.org/pkg/multirow
\DeclareFontEncoding{OT2}{}{} % to enable usage of cyrillic fonts

\usepackage{lipsum} %package for the address 

\usepackage[usenames,dvipsnames]{color}

\usepackage{tikz}

\newtheorem*{theorem*}{Theorem}

\newtheorem{lemma}{Lemma}[section]
\newtheorem{theorem}[lemma]{Theorem}
\newtheorem{proposition}[lemma]{Proposition}

\newtheorem{cor}[lemma]{Corollary}

\newtheorem{claim*}{Claim}
\newtheorem{thm}[lemma]{Theorem}
\newtheorem{defn}[lemma]{Definition}

\theoremstyle{definition}
\newtheorem{remark}[lemma]{Remark}

\newcommand{\PP}{{\mathbb P}}

\newcommand{\F}{{\mathbb F}}

\newcommand{\Z}{{\mathbb Z}}

\newcommand{\curprime}{\mbox{'}}

\DeclareMathOperator{\Aut}{Aut}

\newcommand{\prob}{\mathrm{Prob}}

\numberwithin{equation}{section}
\numberwithin{table}{section}

\title{Statistics for biquadratic covers of the projective line over finite fields.}

\author{Elisa Lorenzo}
\address{Universiteit Leiden\\ 
Mathematisch Instituut\\
Niels Bohrweg 1\\
2333 CA Leiden (The Netherlands)}
\email{e.lorenzo.garcia@math.leidenuniv.nl}

\author{Giulio Meleleo}
\address{Universit\`{a} degli Studi ``Roma Tre''\\
                                Dipartimento di Matematica e Fisica\\
                                Largo San Leonardo Murialdo 1\\
                                00146 Roma (Italy)}
\email{meleleo@mat.uniroma3.it}

\author{Piermarco Milione\\ with an appendix by Alina Bucur}
\address{Universitat de Barcelona\\
                                Departament d'\`{A}lgebra i Geometria\\
                                Gran Via de les Corts Catalanes 585\\
                                08005 Barcelona (Spain)}
\email{pmilione@ub.edu}

\address{University of California, San Diego\\ 
Department of Mathematics\\
9500 Gilman Drive \#0112\\
CA 92093 La Jolla (USA)}
\email{alina@math.ucsd.edu}

\thanks{The second author was partially supported by the ``National Group for Algebraic and Geometric Structures, and their Applications'' (GNSAGA-INDAM). The third author was partially supported by the Spanish Council under project MTM2012-33830 and by an ADR grant from the Universitat de Barcelona}

 \makeatletter
\let\@wraptoccontribs\wraptoccontribs
\makeatother


\begin{document}

%\elisa{} \giulio{} \piermarco{}

\begin{abstract}
We study the distribution of the traces of the Frobenius endomorphism of genus $g$ curves which are quartic non-cyclic covers of $\mathbb{P}_{\mathbb{F}_q}^{1}$, as the curve varies in an irreducible component of the moduli space. We show that for $q$ fixed, the limiting distribution of the trace of Frobenius equals the sum of $q + 1$ independent random discrete variables. We also show that when both $g$ and $q$ go to infinity, the normalized trace has a standard complex Gaussian distribution. Finally, we extend these computations to the general case of arbitrary covers of $\mathbb{P}_{\mathbb{F}_q}^{1}$ with Galois group isomorphic to $r$ copies of $\mathbb{Z}/2\mathbb{Z}$. For $r=1$, we recover the already known hyperelliptic case.
\end{abstract}

\maketitle

\section{Introduction}
One of the most influent result in class field theory is Chebotarev's density theorem. As it is well known, this result is a deep generalization of the Theorem of Dirichlet about equidistribution of rational primes in arithmetic progression and gives a complete understanding of the distribution of primes in a fixed Galois number field extension with respect to their splitting behavior (for an interesting discussion of the theorem and its original proof see \cite{LS96}). 
In the function field case, the parallel statement is carried over by the Sato-Tate conjecture for curves, which studies the distribution of the Frobenius endomorphism of the reduction modulo $p$ of a fixed curve, when the prime $p$ varies. 

In order to complement this research line in other directions, several mathematicians were led to consider the following new general problem:
given a family of curves, of genus $g$ over $\mathbb{F}_{q}$, satisfying certain properties, understand the distribution of the Frobenius endomorphism of the curves of the family. This is sometimes called the \emph{vertical Sato-Tate conjecture}, since the prime $p$ is fixed and the curve varies in the family. We can study the limiting distribution in two different ways, depending on whether we let the genus $g$ or the cardinality $q$ of the field tend to infinity. It is then interesting to compare both limit results. 

When $g$ is fixed and $q$ goes to infinity the problem can be solved thanks to Deligne's equidistribution theorem (cf. \cite{KS99}) while for the complementary case different techniques are applied depending on the particular family considered. The fluctuation in the number of points at the $g$-limit has been studied for different families of curves, such as:
\begin{itemize}
\item 
Hyperelliptic curves , cf. \cite{KR09}, \cite{Betal09},
\item
Cyclic trigonal curves (i.e. cyclic 3-covers of the projective line), cf. \cite{Betal09}, \cite{Xiong2010},
\item
General trigonal curves, cf. \cite{Wood2012},
\item
$p$-fold cover of the projective line, \cite{Betal2011},
\item
$\ell$-covers of the projective line, cf. \cite{Betal09}, \cite{Betal2015}.
\end{itemize}

In the present paper, we study the distribution of the number of points over $\mathbb{F}_{q}$ for a genus $g$ curve $C$ defined over $\mathbb{F}_{q}$ which is a quartic non-cyclic cover of the projective line $\mathbb{P}^{1}_{\mathbb{F}_{q}}$, at the $q$-limit (for a genus $g$ fixed) and at the $g$-limit (with $q$ fixed). This is the first time that a family of non-cyclic abelian covers is studied. The distribution obtained is different to the product of probabilities for the family of hyperelliptic curves, what at first sight could be guessed. Therefore, the study of this family seems to be the first natural step is order to understand the general abelian case. 

Let $\mathcal{B}_{g}(\mathbb{F}_q)$ be the family of genus $g$ quartic non-cyclic cover of the projective line $\mathbb{P}^{1}_{\mathbb{F}_{q}}$, and consider the following decomposition
$$
\mathcal{B}_{g}(\mathbb{F}_q)=\bigcup_{g_{1}+g_{2}+g_{3}=g}\mathcal{B}_{(g_{1},g_{2},g_{3})}(\mathbb{F}_q)
$$
where $\mathcal{B}_{(g_{1},g_{2},g_{3})}(\mathbb{F}_q)$ denotes the subfamily of curves $C\in\mathcal{B}_{g}(\mathbb{F}_q)$ such that the three hyperelliptic quotients of $C$ have genera $g_{1},g_{2}$ and $g_{3}$.

The main theorem of the paper is the following:

\bigskip\noindent {\bf Theorem \ref{main}}\, {\em
If the three genera $g_{1},g_{2},g_{3}$ go to infinity, then we have that  
$$
\frac{|\{C\in\mathcal{B}_{(g_1,g_2,g_3)}(\mathbb{F}_q):\,\mathrm{Tr}(\mathrm{Frob}_C)=-M\}|}{|\mathcal{B}_{(g_1,g_2,g_3)}(\mathbb{F}_q)|}=
\mathrm{Prob}\left(\sum_{j=1}^{q+1}X_{j}=M\right)
$$
where the $X_{j}$ are i.i.d. (identically independently distributed) random variables such that
$$
X_{i}=\left\{\begin{array}{ccc}-1 & \rm{with\,probability}& \frac{3(q+2)}{4(q+3)}\\
\,&\,&\,\\
                                                 1 & \rm{with\,probability}& \frac{6}{4(q+3)}\\
\,&\,&\,\\                                                 
                                              3 & \rm{with\,probability}& \frac{q}{4(q+3)}
\end{array}\right..$$}

\subsection*{Outline}
In Section $2$, we introduce the family of biquadratic curves and we give a parametrization of the family in terms of terns of coprime square-free polynomials. In Section $3$, we compute the monodromy group of the family in the sense of Katz and Sarnak (cf. \cite[Ch. 9]{KS99}) and we obtain the corresponding distribution of the Frobenius traces at the $q$-limit. In Section $4$, previous theorem is proven, and in Section $5$ the moments of the Frobenius traces are computed at the $g$-limit, proving that when both $g$ and $q$ go to infinity the normalized trace has a standard complex Gaussian distribution. In last section, Theorem \ref{main} is generalized for an arbitrary cover of the projective line with Galois group isomorphic to $r$ copies of $\mathbb{Z}/2\mathbb{Z}$. The paper concludes with an Appendix, writen by Alina Bucur, giving the heuristic for the distribution of the number of points for the whole family of $r$-quadratic curves. 

\subsection*{Notations} We now fix some notations and conventions that will be valid in the sequel.
\begin{itemize}
\item
$p\neq 2$ is a prime integer, and $q$ is a positive power of $p$.
\item
$k=\mathbb{F}_{q}(t)$ is the function field of $\mathbb{P}^{1}_{\mathbb{F}_q}$, and $K/k$ is a finite extension.
\item
$(f,g)$ denotes the greatest common divisor of two polynomials $f,g\in\mathbb{F}_q[t]$.
\item
$\text{deg}(f)$ denotes the degree of a polynomial $f$
\item
$|f|:=q^{\text{deg}(f)}$ denotes the norm of a polynomial $f$.
\item
$\tilde{f}$ is the polynomial obtained inverting the order of the coefficients of $f$.
\item
$g(C)$ denotes the geometric genus of the projective curve $C/\mathbb{F}_q$,
\item
and $\text{Frob}_C$ denotes its geometric Frobenius morphism.
\end{itemize}

\subsection*{Acknowledgements} This work was started in March 2014 at the Arizona Winter School ``Arithmetic Statistic'' and the authors would like to thank the organizers for creating such a stimulating working environment. Moreover, the authors are grateful to Alina Bucur and Chantal David for initiating them into this rich and appealing field of research in number theory and for several discussions and helpful comments during the preparation of the paper. Finally, we would like to thank Patrick Meisner for carefully reading a first draft of the paper and for his useful comments.

\section{The family of biquadratic curves}\label{sec:family}

We first define and give the basic properties of the family of biquadratic curves. We determine its genus in terms of the equations defining the curves, and we study the irreducible components of the coarse moduli space of biquadratic curves.
 
Recall that if $K/\mathbb{F}_q(t)$ is a finite Galois extension such that $K\cap \bar{\mathbb{F}}_q=\mathbb{F}_q$, then there exists, up to isomorphism, a unique nonsingular projective curve 
$C$ with function field $\mathbb{F}_q(C)=K$, together with a regular morphism $\varphi: C\rightarrow\mathbb{P}^{1}_{\mathbb{F}_{q}}$ defined over $\mathbb{F}_{q}$ (cf. \cite[I,Th. 6.6, Th.6.9]{HarBook}).

\begin{defn}\label{def_biq_curve}
We call biquadratic curve a smooth projective curve $C$, together with a regular morphism $\varphi: C\rightarrow\mathbb{P}^{1}_{\mathbb{F}_{q}}$ defined over $\mathbb{F}_{q}$, that induces a field extension with Galois group $\mathrm{Gal}(\mathbb{F}_q(C)/\mathbb{F}_q(t))\simeq \mathbb{Z}/2\mathbb{Z}\times\mathbb{Z}/2\mathbb{Z}$. 
\end{defn}

Since $\mathrm{char}(k)\neq 2$, it is clear that every non-cyclic quartic extension of $k$ is of the form $K=k(\sqrt{h_1(t)},\sqrt{h_2(t)})$, for some $h_1(t),h_{2}(t)\in\mathbb{F}_q[t]$ different non-constant polynomials, that we can take to be square-free. Moreover, if the leading coefficient of $h_{i}$ is a square in $\mathbb{F}_{q}$, then we can assume that this is equal to $1$. Therefore, if $C$ is a biquadratic curve, then an affine model of $C$ in $\mathbb{A}^{3}_{\mathbb{F}_{q}}$ is given by  
$$
C:\begin{cases} 
y_{1}^{2}=h_1(t) \\ 
y_{2}^{2}=h_2(t)
\end{cases}.
$$

\begin{remark}\label{coprime}
If $K:=k(\sqrt{h_1(t)},\sqrt{h_2(t)})$ is a biquadratic extension of $k$, then there are exactly 3 different quadratic subextensions of $K$, namely $k(\sqrt{h_{1}}),k(\sqrt{h_{2}})$ and $k(\sqrt{h_{1}h_{2}})$. 

If we write $h_{i}=f_{i}f$ for $i=1,2$, with $f=(h_1,h_2)$, then clearly we have that $(f_{1},f_{2})=(f_1,f)=(f_2,f)=1$ and these three subextensions are $k(\sqrt{ff_{1}}),k(\sqrt{ff_{2}})$ and $k(\sqrt{f_{1}f_{2}})$.

Two such extensions $k(\sqrt{h_1(t)},\sqrt{h_2(t)})$ and $k(\sqrt{h'_1(t)},\sqrt{h'_2(t)})$ define the same biquadratic extension if and only if we have the equality of sets
$$
\{h_1,h_2,\frac{h_1h_2}{(h_1,h_2)^2}\}=\{h'_1,h'_2,\frac{h'_1h'_2}{(h'_1,h'_2)^2}\}.
$$

\end{remark}

\begin{remark}\label{ram_infty}
Recall that if $\pi:C\rightarrow\mathbb{P}^{1}$ is a degree $2$ regular cover, whose affine plane model is $y^{2}=F(t)$, with $F(t)$ a square-free polynomial over $\mathbb{F}_{q}$, then the point at infinity is ramified in the cover $\pi$ if and only if the degree $d$ of $F$ is odd. Indeed, if we take take $u=\frac{1}{t}$, then the function field of $C$ is
$$
k(C)=k(\sqrt{F(t)})=k(\sqrt{F(1/u)})=k(\sqrt{u^{-d}\tilde{F}(u)})
$$
and then it is clear that $t=\infty$ ramifies if and only if the point $u=0$ ramifies, i.e. if and only if $d$ is odd.
\end{remark}

\begin{proposition} \label{genus}
Let $h_1(t),h_{2}(t)\in\mathbb{F}_q[t]$ be different square-free polynomials, and let $C$ be the curve whose function field is $k(C)=k(\sqrt{h_1(t)},\sqrt{h_2(t)})$. For every $i=1,2$, write $h_i=ff_i$, with $f=(h_{1},h_{2})$, and define $h_{3}:=f_{1}f_{2}$.

If we denote by $C_{i}$ the hyperelliptic curve whose affine plane model is given by the equations $y^{2}=h_{i}(t)$, for $i=1,2,3$, then we have the following formula for the genus of $C$:
$$
g(C)=g(C_{1})+g(C_{2})+g(C_{3}).
$$
Moreover, if we denote by $n:=\operatorname{deg}(f)$ and $n_i:=\operatorname{deg}(f_i)$, 
$$
g(C)=g(n_1,n_2,n):=n_1+n_2+n+e_\infty-4,
$$
where $e_{\infty}$ is the ramification index at the point at infinity, that is,
$$
e_{\infty}:=\left\{\begin{array}{ll}2, & \text{if}\;n\equiv n_{1}\equiv n_{2}\equiv 0\,(\mathrm{mod}\,2)\\
1, &otherwise
\end{array} \right..
$$
\end{proposition}

\begin{proof}
Let us denote by $R:=\mathrm{Ram}(\pi)$ the subset of all points of $\mathbb{P}^{1}_{\mathbb{F}_q}$ which are ramified in the cover $\pi: C\longrightarrow\mathbb{P}^{1}_{\mathbb{F}_q}$. Riemann-Hurwitz's formula (cf. \cite[Theorem 7.16]{Rosen02}) implies that $2g(C)-2=4(2\cdot 0-2)+2|R|$. That is, $g(C)=|R|-3$. Again, for the hyperelliptic cover 
$\pi_{i}:C_{i}\longrightarrow\mathbb{P}^{1}$ and the ramification sets $R_{i}:=\mathrm{Ram}(\pi_{i})$, we get $g(C_{i})=\frac{|R_{i}|}{2}-1$. Now, the definition of $h_3$ implies that
$$
2|R_{1}\cup R_{2}\cup R_{3}|=|R_{1}|+|R_{2}|+|R_{3}|.
$$
Thus, the formula $g(C)=g(C_{1})+g(C_{2})+g(C_{3})$ holds.

We can also apply Riemann-Hurwitz's formula to the morphism $\pi$, and so we have
$$
2g-2=4(2\cdot 0-2)+2\cdot(n_1+n_2+n_3+e_{\infty}-1).
$$
\end{proof}

Now, we introduce some sets of polynomials that will be useful:

$$
V_d=\{ F\in \mathbb{F}_q[t]:\,F\,\text{monic},\,\text{deg}(F)=d\},
$$
$$
\mathcal{F}_d=\{ F\in \mathbb{F}_q[t]:\,F\,\text{monic},\,\text{square-free},\,\text{deg}(F)=d\},
$$
$$
\widehat{\mathcal{F}}_d=\{ F\in \mathbb{F}_q[t]:\,F\,\text{square-free},\,\text{deg}(F)=d\},
$$
$$
\mathcal{F}_{(n,n_1,n_2)}=\{(f,f_1,f_2)\in \mathcal{F}_n\times \mathcal{F}_{n_1}\times \mathcal{F}_{n_2}:\,(f,f_1)=(f,f_2)=(f_1,f_2)=1\},
$$
$$
\widehat{\mathcal{F}}_{(n,n_1,n_2)}=\{(f,f_1,f_2)\in \mathcal{F}_n\times \widehat{\mathcal{F}}_{n_1}\times \widehat{\mathcal{F}}_{n_2}:\,(f,f_1)=(f,f_2)=(f_1,f_2)=1\},
$$
$$
\mathcal{F}_{[n,n_1,n_2]}=\mathcal{F}_{(n,n_1,n_2)}\cup \mathcal{F}_{(n-1,n_1,n_2)} \cup \mathcal{F}_{(n,n_1-1,n_2)}\cup \mathcal{F}_{(n,n_1,n_2-1)},
$$
$$
\widehat{\mathcal{F}}_{[n,n_1,n_2]}=\widehat{\mathcal{F}}_{(n,n_1,n_2)}\cup \widehat{\mathcal{F}}_{(n-1,n_1,n_2)} \cup \widehat{\mathcal{F}}_{(n,n_1-1,n_2)}\cup \widehat{\mathcal{F}}_{(n,n_1,n_2-1)}.
$$

\begin{defn} We denote by $\mathcal{B}_g(\mathbb{F}_{q})$ the family of biquadratic curves defined over $\mathbb{F}_{q}$ and of fixed genus $g$. It can be written as a disjoint union of subfamilies indexed by unordered $3$-tuples of positive integers $g_1,g_2,g_3$, i.e. 
$$
\mathcal{B}_g(\mathbb{F}_{q})=\bigcup_{g_1+g_2+g_3=g }\mathcal{B}_{(g_1,g_2,g_3)}(\mathbb{F}_{q}),
$$ 
where $\mathcal{B}_{(g_1,g_2,g_3)}(\mathbb{F}_{q})$ denotes the subfamily of biquadratic curves of genus $g=g_1+g_2+g_3$ such that the intermediate curves given by the morphism to $\mathbb{P}^1$ have genus equal to $g_1,g_2$ and $g_3$. This family is in bijection with the family of curves defined by elements in the set of polynomials $\widehat{\mathcal{F}}_{[n,n_{1},n_{2}]}$ such that $g_{i}=\lfloor \frac{n+n_{i}-1}{2}\rfloor$ for $i=1,2$ and $g_{3}=\lfloor \frac{n_1+n_2-1}{2}\rfloor$.

\end{defn}

The family $\mathcal{B}_{g}(\bar{\mathbb{F}}_{q})$ of biquadratic curves defined over $\bar{\mathbb{F}}_{q}$ is a coarse moduli space over $\mathbb{Z}[1/2]$ (cf.  \cite[Lemma 3.1]{Pries2005}). A detailed geometric study of this moduli space can be found in  \cite{Pries2005} and \cite{Pries2005b}.

\begin{remark} One has the following equalities:
$$
|\mathcal{B}_{(g_1,g_2,g_3)}(\mathbb{F}_{q})|\curprime=\sideset{}{\curprime}\sum_{\scriptscriptstyle C\in \mathcal{B}_{(g_1,g_2,g_3)}(\mathbb{F}_{q})} 1=\sum_{\scriptscriptstyle F\in \widehat{\mathcal{F}}_{[n,n_{1},n_{2}]}}\frac{1}{|\text{Aut}(C)|}=\frac{|\widehat{\mathcal{F}}_{[n,n_{1},n_{2}]}|}{q(q^2-1)},
$$
where the $'$ notation, applied both to cardinality and summation, means that each one of the curves $C$ in the moduli spaces is counted with
the usual weight $\frac{1}{|\text{Aut}(C)|}$.
\end{remark} 

\begin{remark}\label{symmetry} Notice that $|\widehat{\mathcal{F}}_{(n,n_1,n_2)}|=(q-1)^2|\mathcal{F}_{(n,n_1,n_2)}|$ and that we can see the set $\widehat{\mathcal{F}}_{(n,n_1,n_2)}$ as the set of the quadratic twists of elements in $\mathcal{F}_{(n,n_1,n_2)}$ given by the equations
$$
C':\begin{cases}y_{1}^{2}=\alpha_{1} ff_1(t)\\y_{2}^{2}=\alpha_{2} ff_2(t)\end{cases}
$$
where $\alpha_{1},\,\alpha_{2}\in\mathbb{F}_{q}^{*}$.
\end{remark}

\section{Monodromy group of the family}
A useful reference for this section is \cite[Ch. 9]{KS99}.
Let $S$ be an open set of $\mathrm{Spec}\,\mathbb{F}_{q}$ and let $\mathcal{C}\rightarrow S$ be a smooth proper morphism of schemes such that the geometric fibers $C_{x}\otimes\overline{\mathbb{F}}_{q}$ are smooth projective curves of genus $g$ over $\overline{\mathbb{F}}_{q}$. 

Fix a prime integer $\ell\neq p$. Then, there exists an $\ell$-adic representation 
$$
\rho_{\ell}:\pi_{1}(S)\longrightarrow\mathrm{GL}_{2g}(\overline{\mathbb{Q}}_{\ell})
$$
with the following interpolation property:
for every closed point $x:\mathrm{Spec}\,\mathbb{F}_{q}\longrightarrow S$ the induced representation
$$
\mathrm{Gal}(\overline{\mathbb{F}}_{q}/\mathbb{F}_{q})\simeq\pi_{1}(\mathrm{Spec}\,\mathbb{F}_{q})\longrightarrow\pi_{1}(S)
\longrightarrow\mathrm{GL}_{2g}(\overline{\mathbb{Q}}_{\ell})
$$
is isomorphic to the $\ell$-adic representation
$$
\rho_{C_{x},\ell}:\mathrm{Gal}(\overline{\mathbb{F}}_{q}/\mathbb{F}_{q})\longrightarrow\mathrm{Aut}(\mathrm{H}_{et}^{1}(C_{x}\otimes_{\mathbb{F}_{q}}\overline{\mathbb{F}}_{q},\overline{\mathbb{Q}}_{\ell}))\simeq\mathrm{GL}_{2g}(\overline{\mathbb{Q}_{\ell}}).
$$

Once an embedding $\iota:\overline{\mathbb{Q}}_{\ell}\hookrightarrow\mathbb{C}$ is fixed, we have a $2g$-dimensional complex representation $\iota\cdot\rho_{\ell}$. The image of this representation is a subgroup of $\mathrm{GL}_{2g}(\mathbb{C})$ called the monodromy group of the family.

For every integer $d\geq 1$, the set of polynomials $\mathcal{F}_{d}$ defined in Section \ref{sec:family} can be algebraically realized as a Zariski-open subset of $\mathbb{A}^{d}_{\mathbb{F}_{q}}$. This could be done redefining it in the following way:
$$
\mathcal{F}_{d}:=\{(a_{0},\dots,a_{d-1})\in\mathbb{A}^{d}_{\mathbb{F}_{q}}\mid\;D(a_{0},\dots,a_{d-1})\neq 0\},
$$

where $D:\mathbb{A}^{d}_{\mathbb{F}_{q}}\longrightarrow\mathbb{A}^{1}_{\mathbb{F}_{q}}$ is the continuos function such that $D(a_{0},a_{1},\dots,a_{d-1})$ denotes the discriminant of the monic polynomial $a_{0}+a_{1}t+\dots+t^{d}\in\mathbb{F}_{q}[t]$.

Let $\mathcal{H}_{g}$ denote the family of genus $g$ hyperelliptic curves over $\mathcal{F}_{d}$, whose fiber over the polynomial 
$F\in\mathcal{F}_{d}$ is given by the curve whose affine plane model is $y^{2}=F(t)$. In \cite[10.1]{KS99}, it is proved that the monodromy group either of the family $\mathcal{H}_{g}$ over $\mathcal{F}_{2g+1}$ and of the family $\mathcal{H}_{g}$ over $\mathcal{F}_{2g+2}$ is $G_{geom}=\mathrm{Sp}_{2g}(\mathbb{C})$.

\begin{proposition}
The monodromy group of the family $\mathcal{B}_{g_1,g_2,g_3}(\mathbb{F}_{q})$ is the biggest possible one, namely it is the symplectic group $\mathrm{Sp}_{2g}(\mathbb{C})$.
\end{proposition}  

\begin{proof}
The set of polynomials $\mathcal{F}_{[n,n_1,n_2]}$ defined in Section \ref{sec:family} can be realized as a Zariski-open subset of $\mathbb{A}^{n}_{\mathbb{F}_{q}}\times\mathbb{A}^{n_{1}+1}_{\mathbb{F}_{q}}\times\mathbb{A}^{n_{2}+1}_{\mathbb{F}_{q}}$.

The family of genus $g$ curves over $\mathcal{F}_{[n,n_1,n_2]}$, whose fiber over the $3$-tuple $(f,f_1,f_2)\in \mathcal{F}_{[n,n_1,n_2]}$ is given by the curve whose affine model is $y_{1}^{2}=ff_{1}(t),y_{2}^{2}=ff_{2}(t)$, is exactly the subfamily of genus $g$ biquadratic curves $\mathcal{B}_{g_{1},g_{2},g_{3}}(\mathbb{F}_{q})$ defined in Section \ref{sec:family}.

Let $N:=\mathrm{max}(n,n_2,n_3)$. By the symmetry of the parametrization we can assume for example that $N=n$ and then we fix two square-free polynomials $f_1,f_2$ of degrees $n_1,n_2$ such that $f_1f_2$ is also square-free. Therefore, we can consider the open immersion
$$
\{f\in\mathcal{F}_{n}:\, ff_1,ff_2\;\mathrm{square-free}\}\longrightarrow \mathcal{F}_{[n,n_1,n_2]}:\, f \longmapsto (f,f_1,f_2).
$$
The monodromy group of the family $\mathcal{H}_{g}$ of hyperelliptic curves over this subset of $\mathcal{F}_{n}$ is the same as if we consider the family over all $\mathcal{F}_{n}$. Finally, the monodromy group of the family $\mathcal{B}_{(g_1,g_2,g_3)}(\mathbb{F}_{q})$ can only increase and, after results of \cite[10.1]{KS99}, it is the biggest possible one. 
\end{proof}

Applying Deligne's equidistribution theorem (cf. \cite[9.3,9.2]{KS99}) and random matrix theory (cf. \cite[4]{DS94}), we have the following distribution result at the $q$-limit for the family $\mathcal{B}_{(g_1,g_2,g_3)}(\mathbb{F}_{q})$.

\begin{cor}
Let $g\geq 3$ be a fixed integer. When $q$ goes to $\infty$, the classes of the Frobenius automorphisms $\{\mathrm{Frob}_{C}\}_{C\in\mathcal{B}_{(g_1,g_2,g_3)}(\mathbb{F}_{q})}$ acting on the first \'etale cohomology group $\mathrm{H}^{1}_{\text{\'et}}(C,\mathbb{Q}_{\ell})$ are equidistributed with respect to the Haar mesure associated to the maximal compact subgroup of $\mathrm{Sp}_{2g}(\mathbb{C})$, i.e.
$$
\lim_{q\to\infty}\langle\mathrm{Tr}\,\mathrm{Frob}^{m}_{C}\rangle=\left\{\begin{array}{ll} 2g & m=0\\ -\eta_{r} & 1\leq|m|\leq 2g \\ 0 & |m|>2g
\end{array}\right.
$$
where
$$
\eta_{m}:=\left\{\begin{array}{ll} 1 & m\;\mathrm{even}\\ 0 & m\;\mathrm{odd}
\end{array}\right..
$$
\end{cor}

\section{The number of points over $\mathbb{F}_{q}$}\label{distributionFq}
Let $\chi$ denote the quadratic character in $\mathbb{F}_q$. We set, for any element $(f,f_1,f_2)$ in $\widehat{\mathcal{F}}_{(n,n_1,n_2)}$, 
$$
S(f,f_1,f_2)=\sum_{x\in\mathbb{F}_q}(\chi(f\cdot f_1(x))+\chi(f\cdot f_2(x))+\chi(f_1\cdot f_2(x))) \text{, and}
$$
$$
\widehat{S}(f,f_1,f_2)=\sum_{x\in\mathbb{P}^{1}(\mathbb{F}_q)}(\chi(f\cdot f_1(x))+\chi(f\cdot f_2(x))+\chi(f_1\cdot f_2(x))),
$$
where for the point at infinity we define
$$
\chi(F(\infty))=\begin{cases}0&\text{deg}(F)\;\text{odd}\\
                                              1&\text{deg}(F)\;\text{even, leading coefficient is a square in }\mathbb{F}_q\\
                                               -1&\text{deg}(F)\;\text{even, leading coefficient is not a square in }\mathbb{F}_q\end{cases}.
$$
Then, for a curve $C\in\mathcal{B}_{(g_1,g_2,g_3)}(\mathbb{F}_q)$ defined by a $3$-tupla $(f,f_1,f_2)$ we have that
$$
\#C(\mathbb{F}_q)=q+1+\widehat{S}(f,f_1,f_2).
$$
Hence, we have the equality
$$
\frac{|\{C\in\mathcal{B}_{(g_1,g_2,g_3)}(\mathbb{F}_q):\,\text{Tr}(\text{Frob}_C)=-M\}|\curprime}{|\mathcal{B}_{(g_1,g_2,g_3)}(\mathbb{F}_q)|\curprime}=
\frac{|\{(f,f_1,f_2)\in\widehat{\mathcal{F}}_{[n,n_1,n_2]}:\,\widehat{S}(f,f_1,f_2)=M\}|}{|\widehat{\mathcal{F}}_{[n,n_1,n_2]}|}.
$$

The goal of this section is to prove the following theorem.

\begin{thm}\label{main} If the three degrees $n,n_{1},n_{2}$ go to infinity, then we have 
$$
\frac{|\{(f,f_1,f_2)\in\widehat{\mathcal{F}}_{[n,n_1,n_2]}:\,\widehat{S}(f,f_1,f_2)=M\}|}{|\widehat{\mathcal{F}}_{[n,n_1,n_2]}|}=
\mathrm{Prob}\left(\sum_{j=1}^{q+1}X_{j}=M\right),
$$
where the $X_{j}$ are i.i.d. random variables such that
$$
X_{i}=\left\{\begin{array}{ccc}-1 & \rm{with\,probability}& \frac{3(q+2)}{4(q+3)}\\
\,&\,&\,\\
                                                 1 & \rm{with\,probability}& \frac{6}{4(q+3)}\\
\,&\,&\,\\                                                 
                                              3 & \rm{with\,probability}& \frac{q}{4(q+3)}
\end{array}\right..$$
More precisely,
$$
\frac{|\{(f,f_1,f_2)\in\widehat{\mathcal{F}}_{[n,n_1,n_2]}:\,\widehat{S}(f,f_1,f_2)=M\}|}{|\widehat{\mathcal{F}}_{[n,n_1,n_2]}|}
=\mathrm{Prob}\left(\sum_{j=1}^{q+1}X_{j}=M\right)
\left(1+ O(q^{-\frac{(1-\epsilon)}{2}\text{min}(n,n_1,n_2)+q})\right).
$$
\end{thm}

The proof of this Theorem runs similarly to the proof of the equivalent statement for hyperelliptic curves (resp. $l-$cyclic covers) in \cite{KR09} (resp. \cite{Betal09}).

\begin{lemma}\label{S} $($\cite[Lemma 4.2]{Betal09}$)$ For $0\leq l\leq q$, let $x_1,...,x_l$ be distinct elements of $\mathbb{F}_q$. Let $U\in\mathbb{F}_q[t]$ be such that $U(x_i)\neq 0$ for $i=0,...,l$. Let $a_1,...,a_l$ be elements of $\mathbb{F}_{q}^{*}$.Then the cardinality of the set 
$$
\mathcal{S}^{U}_{d}(a_{1},\dots,a_{l}):=\{F\in\mathcal{F}_d:\,(F,U)=1,\,F(x_i)=a_i,\,1\leq i\leq l\}
$$
is the number
$$
S^{U}_{d}(l)=\frac{q^{d}}{\zeta_q(2)}\left(\frac{q}{q^{2}-1}\right)^{l}\prod_{P|U}(1+|P|^{-1})^{-1}\left(1+O(q^{l-d/2})\right).
$$
\end{lemma}

\begin{lemma}\label{guay}  For $0\leq l\leq q$ let $x_1,...,x_l$ be distinct elements of $\mathbb{F}_q$. Let $U\in\mathbb{F}_q[t]$ be such that $U(x_i)\neq 0$ for $i=0,...,l$. Let  $a_1,...,a_l,b_1,...,b_l$ be elements of $\mathbb{F}_{q}^{*}$. Then the cardinality of the set 
$$
\mathcal{R}^{U}_{n_1,n_2}(a_{1},\dots,a_{l},b_{1},\dots,b_{l}):=\{(f_1,f_2)\in\mathcal{F}_{n_1}\times\mathcal{F}_{n_2}:\,(f_i,U)=(f_1,f_2)=1,\,f_1(x_i)=a_i,\,f_2(x_i)=b_i, \,1\leq i\leq l \}
$$
is the number
$$
R^{U}_{n_1,n_2}(l)=\frac{q^{n_1+n_2}L}{\zeta_{q}^{2}(2)}\left(\frac{q}{(q-1)^{2}(q+2)}\right)^{l}\prod_{P|U}\left(\frac{1}{1+2|P|^{-1}}\right)\left(1+O(q^{l-\frac{\text{min}(n_1,n_2)}{2}})\right),
$$
where $L:= \prod_{P\,\text{prime}}(1-\frac{|P|^{-2}}{(1+|P|^{-1})^2})$.
\end{lemma}

\begin{proof} By inclusion-exclusion principle (same notations as in \cite[Theorem 13.5]{CombBook}), with
$$
f(D)=|\{(f_1,f_2)\in\mathcal{F}_{n_1}\times\mathcal{F}_{n_2}:\,(f_i,U)=1,\,D|(f_1,f_2),\,f_1(x_i)=a_i,\,f_2(x_i)=b_i, \,1\leq i\leq l \}|,
$$
$$
g(D)=|\{(f_1,f_2)\in\mathcal{F}_{n_1}\times\mathcal{F}_{n_2}:\,(f_i,U)=1,\,(f_1,f_2)=D,\,f_1(x_i)=a_i,\,f_2(x_i)=b_i, \,1\leq i\leq l \}|,
$$
where $D$ is a polynomial in $\mathbb{F}_q[x]$, we have
$$
R^{U}_{n_1,n_2}(l)=g(1)=\sum_{\scriptscriptstyle D,\,D(x_i)\neq 0,(D,U)=1}\mu (D) f(D).
$$

But notice that when $(D,U)=1$
$$
f(D)=|\{(f_1,f_2)\in\mathcal{F}_{n_1-\text{deg}(D)}\times\mathcal{F}_{n_2-\text{deg}(D)}:\,(f_i,UD)=1,\,f_1(x_i)=a_i,\,f_2(x_i)=b_i, \,1\leq i\leq l\}|,
$$
hence Lemma \ref{S} implies
\begin{multline*}
f(D)=S^{UD}_{n_1-\text{deg}(D)}(l)\cdot S^{UD}_{n_2-\text{deg}(D)}(l)=\\
=\frac{q^{n_1+n_2-2\text{deg}(D)}}{\zeta_{q}^{2}(2)}\left(\frac{q}{q^{2}-1}\right)^{2l}\prod_{P|UD}(1+|P|^{-1})^{-2}\left(1+O(q^{l+\frac{\text{deg}(D)}{2}-\frac{\text{min}(n_1,n_2)}{2}})\right).
\end{multline*}

So, one has
\begin{multline*}
R^{U}_{n_1,n_2}(l)=\sum_{\scriptscriptstyle D,\,D(x_i)\neq 0,(D,U)=1}\mu (D) f(D)=\\
\frac{q^{n_1+n_2}}{\zeta_{q}^{2}(2)}\left(\frac{q}{q^{2}-1}\right)^{2l}\prod_{P|U}(1+|P|^{-1})^{-2}\sum_{\begin{array}{c}\scriptscriptstyle D(x_i)\neq 0,\,(D,U)=1\\\scriptscriptstyle \text{deg}(D)\leq \text{min}(n_1,n_2)\end{array}}\mu(D)|D|^{-2}\prod_{P|D}(1+|P|^{-1})^{-2}\left(1+O(q^{l-\frac{\text{min}(n_1,n_2)}{2}})\right).
\end{multline*}

Now, we observe that
\begin{multline*}
\sum_{\begin{array}{c}\scriptscriptstyle D(x_i)\neq 0,\, (D,U)=1\\ \scriptscriptstyle\text{deg}(D)\leq \text{min}(n_1,n_2)\end{array}}\mu(D)|D|^{-2}\prod_{P|D}(1+|P|^{-1})^{-2}=\\
\sum_{\scriptscriptstyle D,\,D(x_i)\neq 0,\,(D,U)=1}\mu(D)|D|^{-2}\prod_{P|D}(1+|P|^{-1})^{-2}+O(q^{-2\text{min}(n_1,n_2)}),
\end{multline*}
where we have that
\begin{multline*}
\sum_{\scriptscriptstyle D,\,D(x_i)\neq 0,\,(D,U)=1}\mu(D)|D|^{-2}\prod_{P|D}(1+|P|^{-1})^{-2}=\\
=\left(\frac{(q+1)^{2}}{q(q+2)}\right)^{l}\prod_{P|U}\left(\frac{1+2|P|^{-1}}{(1+|P|^{-1})^2}\right)^{-1}\prod_{P\,\text{prime}}\left(1-\frac{|P|^{-2}}{(1+|P|^{-1})^2}\right)
=\left(\frac{(q+1)^{2}}{q(q+2)}\right)^{l}\prod_{P|U}\left(\frac{1+2|P|^{-1}}{(1+|P|^{-1})^2}\right)^{-1}L.
\end{multline*}
We can prove that $0<L<1$ (see next Remark \ref{non_vanishing_L}).
So, finally
$$
R^{U}_{n_1,n_2}(l)=\frac{q^{n_1+n_2}L}{\zeta_{q}^{2}(2)}\left(\frac{q}{(q-1)^{2}(q+2)}\right)^{l}\prod_{P|U}\left(\frac{1}{1+2|P|^{-1}}\right)\left(1+O(q^{l-\frac{\text{min}(n_1,n_2)}{2}})\right).
$$
\end{proof}

\begin{remark}\label{non_vanishing_L}
We need to prove that the infinite product
$\prod_{P\,\text{prime}}(1-\frac{|P|^{-2}}{(1+|P|^{-1})^2})$
converges to a real number $L$ such that $0<L<1$. The Prime Polynomial Theorem implies that this is equivalent to prove that the infinite product
$$\prod_{\nu\geq 1} \left(1-\frac{1} {(q^\nu + 1)^2} \right)^{\frac{q^{\nu}}{\nu}} $$
converges to a positive real number $\tilde{L}$, in particular, we will see that $0<\tilde{L}<1$ (remember that $q\geq 3$).

Because $\left(1-\frac{1} {(q^{\nu}+1)^2} \right)^{\frac{q^{\nu}}{\nu}}<1$, we have that $\tilde{L}<1$. In order to prove that $0<\tilde{L}$, and since for $z\in(0,1)$ we have $\log(1-z)\geq\frac{z}{z-1}$, it is enough to prove that
$$ \sum_{\nu\geq 1}\frac{q^{\nu}}{\nu}\frac{\frac{1}{(q^{\nu}+1)^2}}{\frac{1}{(q^{\nu}+1)^2}-1}=-\sum_{\nu\geq 1}\frac{1}{\nu}\cdot\frac{1}{q^\nu+2} $$
is convergent. Indeed, we have
$$0\leq \sum_{\nu\geq 1}\frac{1}{\nu}\cdot\frac{1}{q^{\nu}+2}\leq\sum_{\nu\geq 1}\frac{1}{\nu 3^\nu}=\log\frac{3}{2}.$$
Thus,
$$\prod_{\nu\geq 1}\left(1-\frac{1}{(q^{\nu}+1)^2}\right)^{\frac{q^{\nu}}{\nu}}\geq\frac{2}{3}.$$
\end{remark}

\begin{proposition}\label{final} For $0\leq l\leq q$, let $x_1,...,x_l$ be distinct elements of $\mathbb{F}_q$, and  $a_1,...,a_l,b_1,...,b_l$ be elements of $\mathbb{F}_{q}^{*}$. Then, for any $1>\epsilon > 0$, we have
$$
|\{(f,f_1,f_2)\in\mathcal{F}_{(n,n_1,n_2)}:\,f(x_i)f_1(x_i)=a_i,\,f(x_i)f_2(x_i)=b_i,\,1\leq i\leq l\}|= 
$$
$$
=\frac{KLq^{n_1+n_2+n}}{\zeta_{q}^{3}(2)}\left( \frac{q}{(q-1)^2(q+3)}\right)^{l}(1+O(q^{-(1-\epsilon)n+\epsilon l}+q^{-n-\frac{\text{min}(n_1,n_2)}{2}+l})),
$$
where $K:=\prod_{P}\left(\frac{1+3|P|^{-1}}{(1+|P|^{-1})(1+2|P|^{-1})}\right)$.
\end{proposition}

\begin{proof} First we observe that
$$
|\{(f,f_1,f_2)\in\mathcal{F}_{(n,n_1,n_2)}:\,f(x_i)f_1(x_i)=a_i,\,f(x_i)f_2(x_i)=b_i,\,1\leq i\leq l\}|=
$$
$$=\sum_{ \begin{array}{c}\scriptscriptstyle f\in\mathcal{F}_n\\ \scriptscriptstyle f(x_i)\neq 0\end{array}}\sum_{\begin{array}{c}\scriptscriptstyle f_1\in\mathcal{F}_{n_1}
\\ \scriptscriptstyle f_1(x_i)=a_if(x_i)^{-1}\\ \scriptscriptstyle (f,f_1)=1\end{array}}\sum_{\begin{array}{c} \scriptscriptstyle f_2\in\mathcal{F}_{n_1}\\ \scriptscriptstyle f_2(x_i)=b_if(x_i)^{-1}\\\scriptscriptstyle (ff_1,f_2)=1\end{array}} 1=
$$
$$
=\sum_{\scriptscriptstyle f\in \mathcal{F}_n,\,f(x_i)\neq 0} R^{f}_{n_1,n_2}(l).
$$
Using Lemma \ref{guay} we have that
$$
|\{(f,f_1,f_2)\in\mathcal{F}_{n,n_1,n_2}:\,f(x_i)f_1(x_i)=a_i,\,f(x_i)f_2(x_i)=b_i,\,1\leq i\leq l\}|= 
$$
$$
=\frac{q^{n_1+n_2}L}{\zeta_{q}^{2}(2)}\left(\frac{q}{(q-1)^{2}(q+2)}\right)^{l}\sum_{U\in \mathcal{F}_n,U(x_i)\neq 0}\prod_{P|U}\frac{1}{1+2|P|^{-1}}+O(q^{n_1+n_2-\frac{\text{min}(n_1,n_2)}{2}-l})=
$$
$$
=\frac{q^{n_1+n_2}L}{\zeta_{q}^{2}(2)}\left(\frac{q}{(q-1)^{2}(q+2)}\right)^{l}\sum_{\text{deg}(U)=n}c(U)+O(q^{n_1+n_2-\frac{\text{min}(n_1,n_2)}{2}-l}),
$$
where for any polynomial $U$, we define
$$
c(U)=\begin{cases}\mu^2(U)\prod_{P|U}\frac{1}{1+2|P|^{-1}}& U(x_i)\neq 0\\ 0&\text{otherwise}\end{cases}.
$$

In order to evaluate $\sum_{\text{deg}(U)=n}c(U)$, we consider the Dirichlet series
$$
G(w)=\sum_{U}\frac{c(U)}{|U|^w}=\prod_{P,\,P(x_i)\neq 0}(1+\frac{1}{|P|^{w}}\cdot \frac{|P|}{(|P|+2)})=
$$
$$
=\frac{\zeta_q(w)}{\zeta_q(2w)}H(w)(1+\frac{1}{q^{w-1}(q+2)})^{-l},
$$
with
$$
H(w)= \prod_{P}(1-\frac{2}{(1+|P|^w)(|P|+2)}).
$$
Notice that $H(w)$ converges absolutely for $\text{Re}(w)>0$, and $G(w)$ is meromorphic for $\text{Re}(w)>0$ with simple poles at the points $w$ where $\zeta_q(w)=(1-q^{1-w})^{-1}$ has poles, that is, $1+i\frac{2\pi n}{\text{log} q}$. Thus, $G(w)$ has a simple pole at $w=1$ with residue
$$
\frac{K}{\zeta_q(2)\text{log}(q)}\left( \frac{q+2}{q+3}\right)^{l},
$$
where $K=H(1)$.

Using Theorem $17.1$ of \cite{Rosen02}, which is the function field version of the Wiener-Ikehara Tauberian Theorem, we get that
$$
\sum_{\text{deg}(U)=n}c(U)= \frac{K}{\zeta_q(2)}\left( \frac{q+2}{q+3}\right)^{l}q^n+O_q(q^{\epsilon n}),
$$
for all $\epsilon\geq 0$ and where, looking at the proof of the theorem and proceding as in Proposition $4.3$ in \cite{Betal09}, we can exchange $O_q(q^{\epsilon n})$ by $O(q^{\epsilon (n+l)})$.
\end{proof}

\begin{cor}\label{corprefinal} For $0\leq l\leq q$,  let $x_1,...,x_l$ be distinct elements of $\mathbb{F}_q$, and let $a_1,...,a_l,b_1,...,b_l$ be elements of $\mathbb{F}_q$ such that $a_1=...=a_{r_{0}}=b_1=...=b_{r_{0}}=0$, $a_{r_{0}+1}=...=a_{r_{0}+r_1}=0=b_{r_{0}+r_1+1}=...=b_{r_{0}+r_1+r_2}$ and $b_{r_{0}+1},...,b_{r_{0}+r_1},a_{r_{0}+r_1+1},...,a_{r_{0}+r_1+r_2},a_j,b_j\neq 0$ if $j>r_{0}+r_1+r_2=m$. Then, for every $\epsilon >0$, the number
$$
\frac{|\{(f,f_1,f_2)\in\mathcal{F}_{(n,n_1,n_2)}:\,f(x_i)f_1(x_i)=a_i,\,f(x_i)f_2(x_i)=b_i,\,f_1(x_i)f_2(x_i)=c_i,\,1\leq i\leq l\}|}{\mid\mathcal{F}_{(n,n_1n_2)}\mid},
$$
where $f(x_i)^2c_i=a_ib_i$, is equal to
$$
\left(\frac{1}{(q-1)(q+3)}\right)^{m}\left(\frac{q}{(q-1)^2(q+3)}\right)^{l-m}\left(1+ O(q^{-\frac{(1-\epsilon)}{2}\text{min}(n,n_1,n_2)+l})\right).
$$
\end{cor}

\begin{proof} Let us write $f=(x-x_1)...(x-x_{r_{0}})f'$, $f_1=(x-x_{r_{0}+1})...(x-x_{r_{0}+r_1})f'_1$, and $f_2=(x-x_{r_{0}+r_1+1})...(x-x_{r_{0}+r_1+r_2})f'_2$. Now, apply Proposition \ref{final} to the $3$-tupla $(f',f'_1,f'_2)$ and sum.

\end{proof}

\begin{cor}\label{corfinal} With notations as in Corollary \ref{corprefinal}, the number
$$
\frac{|\{(f,f_1,f_2)\in\mathcal{F}_{(n,n_1,n_2)}:\,\chi(f(x_i)f_1(x_i))=e^{1}_i,\,\chi(f(x_i)f_2(x_i))=e^{2}_i,\,\chi(f_1(x_i)f_2(x_i))=e_i,\,1\leq i\leq l\}|}{|\mathcal{F}_{(n,n_1,n_2)}|},
$$
where $e^{1}_i,e^{2}_i,e_i\in\{-1,\,0,\,1\}$, $\chi(f(x_i)^2)e_i=e_{i}^{1}e_{i}^{2}$, and exactly $2m$ of them are equal to zero, is equal to
$$
C_{m}^{l}=\left( \frac{q-1}{2}\right)^m 
\left( \frac{q-1}{2}\right)^{2(l-m)}
\left(\frac{1}{(q-1)(q+3)}\right)^{m}
\left(\frac{q}{(q-1)^2(q+3)}\right)^{l-m}\left(1+ O(q^{-\frac{(1-\epsilon)}{2}\text{min}(n,n_1,n_2)+l})\right)=
$$
$$
=\left(\frac{1}{2(q+3)}\right)^{m} 
\left( \frac{q}{4(q+3)}\right)^{l-m}\left(1+ O(q^{-\frac{(1-\epsilon)}{2}\text{min}(n,n_1,n_2)+l})\right).
$$
\end{cor}

\begin{cor}\label{superfinal} For $0\leq l\leq q$,  let $x_1,...,x_l$ be distinct elements of $\mathbb{P}^1(\mathbb{F}_q)$, and let $e^{1}_i,e^{2}_i,e_i\in\{-1,\,0,\,1\}$ be such that $\chi(f(x_i)^2)e_i=e_{i}^{1}e_{i}^{2}$, where exactly $2m$ of them are equal to zero. Then
$$
\frac{|\{(f,f_1,f_2)\in\widehat{\mathcal{F}}_{[n,n_1,n_2]}:\,\chi(f(x_i)f_1(x_i))=e^{1}_i,\,\chi(f(x_i)f_2(x_i))=e^{2}_i,\,\chi(f_1(x_i)f_2(x_i))=e_i\}|}{|\widehat{\mathcal{F}}_{[n,n_1,n_2]}|}
$$
is also equal to the number $C^{l}_{m}$ defined in Corollary \ref{corfinal}.
\end{cor}

\begin{proof} 
Distinguish the case in which some $x_j$ is the point at infinity or not. Generalize Corollary \ref{corfinal} for the sets $\widehat{\mathcal{F}}_{(n,n_1,n_2)}$ looking at the symmetry observed in Remark \ref{symmetry}, and add for the different components of $\widehat{\mathcal{F}}_{[n,n_1,n_2]}$.
\end{proof}

\begin{proof} (of Theorem \ref{main}) Apply Corollary \ref{superfinal} in order to compute
$$
\frac{|\{(f,f_1,f_2)\in\widehat{\mathcal{F}}_{[n,n_1,n_2]}:\widehat{S}(f,f_1,f_2)=M|}{|\widehat{\mathcal{F}}_{[n,n_1,n_2]}|}=
$$
$$
=\sum_{\scriptscriptstyle \epsilon_1,...,\epsilon_{q+1}\in \{-1,1,3\}, \sum \epsilon_i =M} \sum_{j=0}^{N_{-1}}\binom{N_{-1}}{j}3^{N_1+N_{-1}} C^{q+1}_{N_1+j}=
$$
$$
=\sum_{\scriptscriptstyle \epsilon_1,...,\epsilon_{q+1}\in \{-1,1,3\},\,\sum \epsilon_i =M} \left(\frac{6}{4}\frac{1}{q+3}\right)^{N_{1}}\left(\frac{3}{4}\frac{q+2}{q+3}\right)^{N_{-1}}\left(\frac{1}{4}\frac{q}{q+3}\right)^{N_{3}}\left(1+ O(q^{-\frac{(1-\epsilon)}{2}\text{min}(n,n_1,n_2)+q})\right)
$$
$$
=\text{Prob}\left(\sum_{1}^{q+1}X_i=M\right)\left(1+ O(q^{-\frac{(1-\epsilon)}{2}\text{min}(n,n_1,n_2)+q})\right),
$$
where we use the notation $N_i$ for the number of elements equal to $i$ in the set $\{\epsilon_1,...,\epsilon_{q+1}\}$.

\end{proof}

\section{Averages and moments sequences}

We want to compute the moments of $\text{Tr}(\text{Frob}_C)/\sqrt{1+q}$. That is, the \textit{k}th moments
$$
M_k(q,g_1,g_2,g_3)=\frac{1}{|\mathcal{B}_{(g_1,g_2,g_3)}(\mathbb{F}_{q})|\curprime}\sum_{\scriptscriptstyle C\in \mathcal{B}_{(g_1,g_2,g_3)}(\mathbb{F}_{q})}\hspace{-1cm}\curprime\hspace{0.8cm}\left(\frac{\text{Tr}(\text{Frob}_C)}{\sqrt{1+q}}\right)^k.
$$

\begin{theorem}\label{moments} With notation in Theorem \ref{main}, we have
$$
M_k(q,g_1,g_2,g_3)=\mathbb{E}\left( \left(\frac{1}{\sqrt{1+q}}\sum_{i=1}^{1+q}X_i\right)^k\right)+O(q^{-\frac{1-\epsilon}{2}min(n,n_{1},n_{2})+k}).
$$
\end{theorem}

\begin{cor} If $g_1,g_2,g_3$ and $q$ tend to infinity, then the moments of $\mathrm{Tr}(\mathrm{Frob}_C)/\sqrt{1+q}$, as $C$ runs over the irreducible component $\mathcal{B}_{(g_1,g_2,g_3)}(\mathbb{F}_{q})$ of the moduli space $\mathcal{B}_{g}(\mathbb{F}_{q})$, are asymptotically Gaussian with mean $0$ and variance $1$.
\end{cor}

\begin{proof} Since the moments of a sum of bounded i.i.d. random variables converge to the Gaussian moments (\cite[Sec. 30]{Bill95}), it follows that, as all $q,g_1,g_2,g_3$ go to $\infty$, $M_k(q,g_1,g_2,g_3)$ agrees with Gaussian moments for all $k$. Then, Theorem $30.2$ in \cite{Bill95} implies the corollary.
\end{proof}

\begin{proof} (of Theorem \ref{moments}) We can write the \textit{k}th moment as
$$
M_k(q,g_1,g_2,g_3)=(-1)^k\frac{q(q^2-1)}{|\widehat{\mathcal{F}}_{[2g_1+2,2g_2+2,2g_3+2]}|}\sum_{\scriptscriptstyle (f_1,f_2,f_3)\in\widehat{\mathcal{F}}_{[2g_1+2,2g_2+2,2g_3+2]}}
(\widehat{S}(f_1,f_2,f_3))^k=
$$
$$
=\frac{(-1)^k q(q^2-1)}{|\widehat{\mathcal{F}}_{[2g_1+2,2g_2+2,2g_3+2]}|}\sum_{\scriptscriptstyle (f_1,f_2,f_3)\in\widehat{\mathcal{F}}_{[2g_1+2,2g_2+2,2g_3+2]}}
\sum_{x\in\mathbb{P}^1(\mathbb{F}_q)}
(\chi(f\cdot f_1(x))+\chi(f\cdot f_2(x))+\chi(f_1\cdot f_2(x)))^k=
$$
$$
=(-1)^k q(q^2-1)\sum_{l=1}^{k}c(k,l)\sum_{(x,b)\in P_{k,l}}\frac{1}{|\widehat{\mathcal{F}}_{[2g_1+2,2g_2+2,2g_3+2]}|}
\sum_{\scriptscriptstyle (f_1,f_2,f_3)\in\widehat{\mathcal{F}}_{[2g_1+2,2g_2+2,2g_3+2]}} B(x,b,f_1,f_2,f_3),
$$
where, borrowing the notation in \cite[Sec. 5]{Betal09} 
$$
P_{k,l}=\left\{(x,b):x=(x_1,...,x_l)\in\mathbb{P}^1(\mathbb{F}_q)^l,\,x_i's\,\text{distinct,}\,b=(b_1,...,b_l)\in\mathbb{Z}^{l}_{>0},\,\sum_{i=1}^{l}b_i=k\right\},
$$
$$
 B(x,b,f_1,f_2,f_3)=\prod_{i=1}^{l}(\chi(f\cdot f_1(x_i))+\chi(f\cdot f_2(x_i))+\chi(f_1\cdot f_2(x_i)))^{b_i}
$$
and $c(k,l)$ is a certain combinatorial factor. We do not need exact formulas for it, but as it was notice in \cite{Betal09}
\begin{equation}\label{cota}
\sum_{l=1}^{k}c(k,l)\sum_{(x,b)\in P_{k,l}}1=(q+1)^k.
\end{equation}
Fix a vector $(x,b)\in P_{(k,l)}$. Then, the number
$$
\sum_{\scriptscriptstyle (f_1,f_2,f_3)\in\widehat{\mathcal{F}}_{[2g_1+2,2g_2+2,2g_3+2]}} \frac{B(x,b,f_1,f_2,f_3)}{|\widehat{\mathcal{F}}_{[2g_1+2,2g_2+2,2g_3+2]}|}=\sum_{\begin{array}{c}\scriptstyle(\epsilon_1,...\epsilon_l)\\ \scriptstyle\epsilon_i\in\{-1,1,3\}\end{array}}\left(\prod P_{\epsilon_i}\right)\left(\prod \epsilon_{i}^{b_i}\right)=
$$
$$
\sum_{\begin{array}{c}\scriptstyle(\epsilon_1,...\epsilon_l)\\ \scriptstyle\epsilon_i\in\{-1,1,3\}\end{array}}\left(\prod P_{\epsilon_i} \epsilon_{i}^{b_i}\right)=\prod_{i}\left(\frac{3^{b_i}q+6+(-1)^{b_i}3(q+2)}{4(q+3)}\right)
(1+O(q^{-\frac{1-\epsilon}{2}min(n,n_{1},n_{2})+l}))
$$
We obtain that
$$
M_k(q,g_1,g_2,g_3)=(-1)^k q(q^2-1)\sum_{l=1}^{k}c(k,l)\sum_{(x,b)\in P_{k,l}}\prod_{i}\left(\frac{3^{b_i}q+6+(-1)^{b_i}3(q+2)}{4(q+3)}\right)
(1+O(q^{-\frac{1-\epsilon}{2}min(n,n_{1},n_{2})+k})).
$$
where the error term is estimated using \ref{cota}.

On the other hand, the corresponding moment of the normalized sum of our random variables is
$$
\mathbb{E}\left( \left(\frac{1}{\sqrt{1+q}}\sum_{i=1}^{1+q}X_i\right)^k\right)=\frac{1}{(1+q)^{k/2}}\sum_{l=1}^k
\sum_{(i,b)\in A_{k,l}}\mathbb{E}(X_{i_1}^{b_1}...X_{i_l}^{b_l}),$$
where
$$
A_{k,l}=\left\{(i,b):i=(i_1,...,i_l),1\leq i_j\leq q+1,\,i_{j}'s\,\text{distinct,}\,b=(b_1,...,b_l)\in\mathbb{Z}^{l}_{>0},\,\sum_{i=1}^{l}b_i=k\right\}
$$
is clearly isomorphic to $P_{k,l}$.

Since
$$
\mathbb{E}(X_{i}^{b})=\frac{3^bq+6+(-1)^b3(q+2)}{4(q+3)}
$$
and $X_1,...,X_{1+q}$ are independent, we get the equality in the statement of the theorem.
\end{proof}

\section{General case: the family of $r$-quadratic curves}

\begin{defn}
Let $r\geq 1$ be an integer. We call $r$-quadratic curve a  non-singular projective curve $C/\mathbb{F}_q$ together with a morpshim $\varphi: C\longrightarrow \mathbb{P}^1_{\mathbb{F}_q}$ defined over $\mathbb{F}_{q}$ such that it induces a function field extension with Galois group 
$$
\mathrm{Gal}(\mathbb{F}_q(C)/\mathbb{F}_q(t))\simeq (\mathbb{Z}/2\mathbb{Z})^{r} 
%\begin{matrix} 
%    \,\\
%      \underbrace{ \mathbb{Z}/2\mathbb{Z}\times\dots\times\mathbb{Z}/2\mathbb{Z} } \\
%      r
%\end{matrix}.
$$

Note that when $r=1$ and $r=2$ we find respectively the definition of hyperelliptic curve and that of biquadratic curve given in Section \ref{sec:family}.
\end{defn}

The family of $r$-quadratic curves is studied in \cite{Pries2005b} where it is proved that when we consider the family of this curve defined over $\bar{\mathbb{F}}_{q}$ we obtain a course moduli space over $\mathbb{Z}[1/2]$. In the same paper a formula for the genus of an $r$-quadratic curve is also proved.

\begin{proposition} 
Let $r\geq 1$ be an integer and let $C/\mathbb{F}_{q}$ be an $r$-quadratic curve.

\begin{enumerate}
\item
An affine model of $C$ in $\mathbb{A}^{r+1}$ is given by
$$
C:\begin{cases}y_{1}^{2}=h_1(t)\\...\\y_{r}^{2}=h_r(t)\end{cases}
$$
where each $h_{i}$ is square-free and different (up to squares) from $\prod_{j\in J}h_{j}$, for every non-empty subset $J\subseteq\{1,\dots,r\}, J\neq \{i\}$. 

\item
For every non-empty $J\subseteq\{1,\dots,r\}$, the affine equation
$$
y^{2}=\prod_{j\in J}h_{j}(t)
$$
defines a quadratic subextension of $\mathbb{F}_q(C)/k$ and every quadratic subextension of $\mathbb{F}_q(C)/k$ is obtained in this way, so there are $2^r-1$ of them.

\item
If we write $C_J$ for the hyperelliptic curve given by the affine equation $y^2=\prod_{J\in J}h_{j}(t)$, then 
$$
g(C)=\sum_{J\subseteq\{1,\dots,r\}}g(C_{J}).
$$
\end{enumerate}
\end{proposition}

Next Theorem gives a convenient way of describing the family of $r$-quadratic curves, and it is the key point to compute the distribution of the Frobenius traces.

\begin{theorem}\label{family}
There exists a one-to-one correspondence between the 
set of $r$-quadratic extensions of $k$ and the set of unordered $2^{r}-1$-tuples of square-free and pairwise coprime polynomials. 
\end{theorem}

\begin{proof}
Let $K=k(\sqrt{h_{1}},\dots,\sqrt{h_{r}})$ be an $r$-quadratic extension. We associate to such an extension, a $2^{r}-1$-tuple $(f_{1},\dots,f_{2^{r}-1})$ of square-free and pairwise coprime polynomials in the following way: for every $i\in\{1,...,2^r-1\}$, we write $B^{i}_{1}\dots B^{i}_{r}$ for the representation of the integer $i$ in base $2$ (so $B^{i}_{j}\in\{1,0\}$ for every $1\leq j\leq r$) and we define $m_{i}$ to be the greatest common divisor of all polynomials $h_{j}$ such that the $B_{j}=1$. 

We then define the polynomials $f_{i}$ as the maximum factor in the decomposition of $m_{i}$ which is coprime to all the $h_{j}$ such that $B^{i}_{j}=0$. Notice that, in particular, $f_{2^{r}-1}=(h_{1},\dots,h_{r})$.

Viceversa, given a tuple $(f_{1},\dots,f_{2^{r}-1})$ of square-free and pairwise coprime polynomials, we define the $r$-quadratic extension $k(\sqrt{h_{1}},\dots,\sqrt{h_{r}})$, where $h_{i}$ is defined to be the product of the $f_{j}$ such that the $i$-th digit of $j$ in base $2$ is $1$, i.e. $B_{i}^{j}=1$.
\end{proof}

Notice that, with notations of Theorem \ref{family}, we have that $f_{2^{r}-1}=(h_{1},\dots,h_{r})$.

After Theorem \ref{family}, we are led to define the following sets:
$$
\mathcal{F}_{(n_1,\dots,n_{2^r-1})}:=\{(f_1,\dots,f_{2^r-1})\in\mathcal{F}_{n_1}\times\dots\times\mathcal{F}_{n_{2^r-1}}:\,(f_{i},f_{j})=1,\;\;i,j=1,\dots,2^{r}-1, i\neq j\},
$$
$$
\widehat{\mathcal{F}}_{(n_1,\dots,n_{2^r-1})}:=\{(f_1,\dots,f_{2^r-1})\in\widehat{\mathcal{F}}_{n_1}\times\dots\times\widehat{\mathcal{F}}_{n_{2^r-2}}\times\mathcal{F}_{n_{2^r-1}}:\,(f_{i},f_{j})=1,\;\;i,j=1,\dots,2^{r}-1,i\neq j\},
$$

It is easy to prove that if $C$ is an $r$-quadratic curve whose affine model is given by equations $y_{i}^{2}=h_{i}(t),\;i=1,\dots,r$, then 
$$
\#C(\mathbb{F}_{q})=\sum_{x\in\mathbb{P}^{1}_{\mathbb{F}_q}}\prod_{i=1}^{r}(1+\chi(h_{i}(x))).
$$
Now, we express this formula in terms of the polynomials $f_{1},\dots,f_{2^r-1}$ defined in the proof of Theorem \ref{family}.

Let us fix $(f_{1},\dots,f_{2^r-1})\in\widehat{\mathcal{F}}_{(n_1,\dots,n_{2^r-1})}$. For every $i\in\{1,\dots,2^r-1\}$, we define the polynomial $p_{i}$ as the square-free part of the product of the polynomials $f_{j}$ such that the relation between the representations in base $2$ of $i$ and $j$ is the following: $B_{k}^{i}=1\Rightarrow B_{k}^{j}=1$. It is immediate to see that, inside the correspondence of Theorem \ref{family}, the square-free polynomials $p_{1},\dots,p_{2^r-1}$ define all the quadratic subextensions of the $r$-quadratic extension defined by the tuple $(f_{1},\dots,f_{2^r-1})$. 

Then we define
$$
\widehat{S}(f_{1},\dots,f_{2^r-1}):=\sum_{x\in\mathbb{P}^{1}}\sum_{i=1}^{2^r-1}\chi(p_{i}(x))
$$
and we can rewrite
$$
\#C(\mathbb{F}_{q})=\sum_{x\in\mathbb{P}^{1}}\prod_{i=1}^{2^r-1}(1+\chi(p_{i}(x)))=q+1+\widehat{S}(f_{1},\dots,f_{2^r-1}).
$$
When $r=2$ we find the formula of Section \ref{distributionFq}.

\begin{lemma}\label{guay_gen} 
Let $n_{1},\dots,n_{\beta}$ be positive integers. For $0\leq l\leq q$, let $x_1,...,x_l$ be distinct elements of $\mathbb{F}_q$. Let $U\in\mathbb{F}_q[X]$ be such that $U(x_i)\neq 0$ for $i=0,...,l$. Let be $a^{1}_1,...,a^{1}_l,\dots,a^{\beta}_1,...,a^{\beta}_l\in\mathbb{F}_{q}^{*}$. The number of elements in the set
\begin{multline*}
\mathcal{R}^{U}_{n_{1},\dots,n_{\beta}}((a_{1}^{j},\dots,a_{l}^{j})_{1\leq j\leq\beta}):=\{(f_1,\dots,f_{\beta})\in\mathcal{F}_{n_1}\times\dots\times\mathcal{F}_{n_{\beta}}:\,\\(f_j,U)=1,\,(f_{j},f_{k})_{k\neq j}=1,\,f_j(x_i)=a^{j}_i,1\leq i\leq l,1\leq j,k\leq\beta\}
\end{multline*}
is the number
$$
R^{U}_{n_{1},\dots,n_{\beta}}(l)=\frac{q^{n_{1}+\dots+n_{\beta}}L_{\beta}}{\zeta_{q}^{\beta}(2)}\left(\frac{q}{(q-1)^{\beta}(q+\beta)}\right)^{l}\prod_{P\mid U}\left(\frac{1}{1+\beta|P|^{-1}}\right)(1+O(q^{l-\frac{\mathrm{min}(n_1,\dots,n_{\beta})}{2}})),
$$
where the constant
$$
L_{\beta}:=
\prod_{P\,\text{prime}}\left(\frac{|P|^{\beta-1}(|P|+\beta)}{(|P|+1)^{\beta}}\right).
$$
In a similar way to Remark \ref{non_vanishing_L}, we can see that $L_{\beta}$ is bounded.
\end{lemma}

\begin{proof}
We will prove it by induction on $\beta$. We find Lemma \ref{S} for $\beta=1$, and for $\beta=2$ we find Lemma \ref{guay}. Assume that the equality of the statement is true for $\beta-1$. 

By inclusion-exclusion principle, with
\begin{equation*}
%f(D)=|\{(f_1,...,f_\beta)\in\mathcal{F}_{n_1}\times\dots\times\mathcal{F}_{n_{\beta}}: \,(f_j,U)=1,\,f_j(x_i)=a^{j}_i,1\leq i\leq l,1\leq j,k\leq\beta,\\
%(f_{j},f_{k})_{j\neq k}=1,\,1\leq j,k\leq\beta-1,\,D|(f_1...f_{\beta-1},f_\beta)\}|,
f(D)=|\{(f_1,...,f_\beta)\in \mathcal{R}^{U}_{n_{1},\dots,n_{\beta-1}}((a_{1}^{j},\dots,a_{l}^{j})_{1\leq j\leq\beta-1})\times \mathcal{S}_{n_\beta}^{U}(a_{1}^{\beta},\dots,a_{l}^{\beta}):
\,D|(f_1...f_{\beta-1},f_\beta)\}|,
\end{equation*}

\begin{equation*}
%g(D)=|\{(f_1,...,f_\beta)\in\mathcal{F}_{n_1}\times\dots\times\mathcal{F}_{n_{\beta}}:\,(f_j,U)=1,\,f_j(x_i)=a^{j}_i,1\leq i\leq l,1\leq j,k\leq\beta,\\
%(f_{j},f_{k})_{j\neq k}=1,\,1\leq j,k\leq\beta-1,\,D=(f_1...f_{\beta-1},f_\beta)\}|,
g(D)=|\{(f_1,...,f_\beta)\in \mathcal{R}^{U}_{n_{1},\dots,n_{\beta-1}}((a_{1}^{j},\dots,a_{l}^{j})_{1\leq j\leq\beta-1})\times \mathcal{S}_{n_\beta}^{U}(a_{1}^{\beta},\dots,a_{l}^{\beta}):\,D=(f_1...f_{\beta-1},f_\beta)\}|,
\end{equation*}

we have
$$
R^{U}_{n_{1},\dots,n_{\beta}}(l)=g(1)=\sum_{\scriptscriptstyle D,\,D(x_i)\neq 0,(D,U)=1}\mu (D) f(D).
$$

But notice that when $(D,U)=1$ and $D$ is square-free
$$
f(D)=\prod_{P|D}(\beta-1) \cdot R^{UD}_{n_{1},\dots,n_{\beta-1}-\text{deg}(D)}(l)\cdot S^{UD}_{n_{\beta}-\text{deg}(D)}(l)
$$
hence, by induction hypothesis,
\begin{multline*}
f(D)=\frac{q^{n_{1}+\dots+n_{\beta}}L_{\beta-1}}{\zeta_{q}^{\beta}(2)}\left(\frac{q^{2}}{(q^{2}-1)(q-1)^{\beta-1}(q+\beta-1)}\right)^l\\
\prod_{P|U}\frac{1}{(1+(\beta-1)|P|^{-1})(1+|P|^{-1})}\prod_{P|D}\frac{(\beta-1)|P|^{-2}}{(1+(\beta-1)|P|^{-1})(1+|P|^{-1})}\left(1+O(q^{l+\frac{\text{deg}(D)}{2}-\frac{\text{min}(n_1,...,n_\beta)}{2}})\right).
\end{multline*}
So, one has
\begin{multline*}
R^{U}_{n_{1},\dots,n_{\beta}}(l)=\sum_{\scriptscriptstyle D,\,D(x_i)\neq 0,(D,U)=1}\mu (D) f(D)=\\
=\frac{q^{n_{1}+\dots+n_{\beta}}L_{\beta-1}}{\zeta_{q}^{\beta}(2)}\left(\frac{q^{2}}{(q^{2}-1)(q-1)^{\beta-1}(q+\beta-1)}\right)^l
\prod_{P|U}\frac{1}{(1+(\beta-1)|P|^{-1})(1+|P|^{-1})}\cdot\\
\sum_{\begin{array}{c}\scriptscriptstyle D(x_i)\neq 0,\,(D,U)=1\\\scriptscriptstyle \text{deg}(D)\leq \text{min}(n_1,...,n_\beta)\end{array}}\mu(D)\prod_{P|D}\frac{(\beta-1)|P|^{-2}}{(1+(\beta-1)|P|^{-1})(1+|P|^{-1})}\left(1+O(q^{l-\frac{\text{min}(n_1,...,n_\beta)}{2}})\right).
\end{multline*}
Now, we observe that
\begin{multline*}
\sum_{\begin{array}{c}\scriptscriptstyle D\, D(x_i)\neq 0,\,(D,U)=1\\ \scriptscriptstyle\text{deg}(D)\leq \text{min}(n_1,...,n_\beta)\end{array}}\mu(D)\prod_{P|D}\frac{(\beta-1)|P|^{-2}}{(1+(\beta-1)|P|^{-1})(1+|P|^{-1})}=\\
\sum_{\scriptscriptstyle D,\,D(x_i)\neq 0,\,(D,U)=1}\mu(D)\prod_{P|D}\frac{(\beta-1)|P|^{-2}}{(1+(\beta-1)|P|^{-1})(1+|P|^{-1})}+O(q^{-2\text{min}(n_1,...,n_\beta)}),
\end{multline*}
where we have that
$$
\begin{array}{l}
\sum\limits_{\scriptscriptstyle D,\,D(x_i)\neq 0,\,(D,U)=1}\mu(D)\prod\limits_{P|D}\frac{(\beta-1)|P|^{-2}}{(1+(\beta-1)|P|^{-1})(1+|P|^{-1})}=
\\
\,
\\
=\left( 1-\frac{(\beta-1)q^{-2}}{(1+(\beta-1)q^{-1})(1+q^{-1})}\right)^{-l}
\prod\limits_{P|U}\left( 1-\frac{(\beta-1)|P|^{-2}}{(1+(\beta-1)|P|^{-1})(1+|P|^{-1})}\right)^{-1}\prod\limits_{P\,\text{prime}}\left(1-\frac{(\beta-1)|P|^{-2}}{(1+(\beta-1)|P|^{-1})(1+|P|^{-1})}\right)=
\\
\,
\\
=\left(  \frac{(q+\beta-1)(q+1)}{q(q+\beta)}\right)^{l}\prod\limits_{P|U}\left(  \frac{\beta|P|^{-1}+1}{(1+(\beta-1)|P|^{-1})(1+|P|^{-1})}\right)^{-1}\prod\limits_{P\,\text{prime}} \frac{\beta|P|^{-1}+1}{(1+(\beta-1)|P|^{-1})(1+|P|^{-1})}.
\end{array}
$$

%\begin{multline*}
%\sum_{\scriptscriptstyle D,\,D(x_i)\neq 0,\,(D,U)=1}\mu(D)\prod_{P|D}\frac{(\beta-1)|P|^{-2}}{(1+(\beta-1)|P|^{-1})(1+|P|^{-1})}
%=\left( 1-\frac{(\beta-1)q^{-2}}{(1+(\beta-1)q^{-1})(1+q^{-1})}\right)^{-l}
%\\
%\prod_{P|U}\left( 1-\frac{(\beta-1)|P|^{-2}}{(1+(\beta-1)|P|^{-1})(1+|P|^{-1})}\right)^{-1}\prod_{P\,\text{prime}}\left(1-\frac{(\beta-1)|P|^{-2}}{(1+(\beta-1)|P|^{-1})(1+|P|^{-1})}\right)=
%\\
%=\left(  \frac{(q+\beta-1)(q+1)}{q(q+\beta)}\right)^{l}\prod_{P|U}\left(  \frac{\beta|P|^{-1}+1}{(1+(\beta-1)|P|^{-1})(1+|P|^{-1})}\right)^{-1}\prod_{P\,\text{prime}} \frac{\beta|P|^{-1}+1}{(1+(\beta-1)|P|^{-1})(1+|P|^{-1})}.
%\end{multline*}
So, the result follows.
\end{proof}

\begin{proposition}\label{final_gen} For $0\leq l\leq q$, let $x_1,...,x_l$ be distinct elements of $\mathbb{F}_q$, and let be $a^{1}_1,...,a^{1}_l,\dots,a^{r}_1,...,a^{r}_l\in\mathbb{F}_{q}^{*}$. Then for any $1>\epsilon > 0$, we have
$$
|\{(f_1,\dots,f_{2^r-1})\in\mathcal{F}_{(n_1,\dots,n_{2^r-1})}:\,h_{j}(x_{i})=a^{j}_{i},\;1\leq i\leq l,\,1\leq j\leq r \}|=
$$
$$
=\frac{q^{n_{1}+\dots+n_{2^{r}-1}}L_{2^{r}-2}}{\zeta_{q}^{2^{r}-1}(2)}\left(\frac{q}{(q-1)^{r}(q+2^{r}-1)}\right)^{l}(1+O(q^{l-\frac{\mathrm{min}(n_1,\dots,n_{\beta})}{2}})).
$$
\end{proposition}

\begin{proof} Apply previous Lemma with $\beta=2^r-1$ and $U(x)=1$. Notice that the value of the polynomials $h_i$ can be fixed by controling the value of $r$ of the $f_j$ polynomials, so we need to multiply the previous number by $(q-1)^{l(2^r-1-r)}$.

\end{proof}

\begin{cor}\label{corofinal_gen}For $0\leq l\leq q$, let $x_1,...,x_l$ be distinct elements of $\mathbb{F}_q$, and let $a^{1}_1,...,a^{1}_l,\dots,a^{r}_1,...,a^{r}_l$ be elements in 
$\mathbb{F}_{q}$ such that for exactly $m$ values of $i$ we have $\prod_{i=1}^{r}a_{i}^{j}=0$. Then, for any $1>\epsilon > 0$, we have
$$
|\{(f_1,\dots,f_{2^r-1})\in\mathcal{F}_{(n_1,\dots,n_{2^r-1})}:\,h_{j}(x_{i})=a^{j}_{i},\;1\leq i\leq l,\,1\leq j\leq r \}|=
$$
$$
=\frac{q^{n_{1}+\dots+n_{2^{r}-1}}L_{2^{r}-2}}{\zeta_{q}^{2^{r}-1}(2)}\left(\frac{q}{(q-1)^{r}(q+2^{r}-1)}\right)^{l-m}\left(\frac{1}{(q-1)^{r-1}(q+2^{r}-1)}\right)^{m}
(1+O(q^{l-\frac{\mathrm{min}(n_1,\dots,n_{\beta})}{2}}).
$$
\end{cor}

\begin{proof} Acording to the values of $a^{j}_{i}$ we can decide which of the polynomials $f_k$ should satisfy $f_j(x_i)=0$, see Theorem \ref{family}. Write $f_k(x)=(x-x_i)f'_k(x)$ with $f'_k(x_i)\neq 0$ and apply previous Proposition. Hence for each value of $i$ such that $\prod_{i=1}^{r}a_{i}^{j}=0$, we should multiply the number in Proposition \ref{final_gen} by $(q-1)/q$.
\end{proof}

Finally, previous corollary with $l=q$ together with the fact that there are $\frac{q-1}{2}$ squares in $\mathbb{F}_q$, implies the following generalization of Theorem \ref{main}:

\begin{thm}\label{main_gen}
If the degrees $n_{1},\dots,n_{2^{r}-1}$ go to infinity 
$$
\frac{|\{(f_{1},\dots,f_{2^{r}-1})\in\widehat{\mathcal{F}}_{(n_1,\dots,n_{2^r-1})}:\,\widehat{S}(f_{1},\dots,f_{2^r-1})=M\}|}{|\widehat{\mathcal{F}}_{(n_1,\dots,n_{2^r-1})}|}=
\mathrm{Prob}\left(\sum_{j=1}^{q+1}X_{j}=M\right)
$$
where the $X_{j}$ are i.i.d. random variables such that
$$
X_{i}=\left\{\begin{array}{ccc} -1               & \rm{with\,probability}& \frac{(2^{r}-1)(q+2^{r}-2)}{2^{r}(q+2^{r}-1)}\\
\,&\,&\,\\
                                                 2^{r-1}-1  & \rm{with\,probability}& \frac{2(2^{r}-1)}{2^{r}(q+2^{r}-1)}\\
\,&\,&\,\\                                                 
                                                 2^{r}-1        & \rm{with\,probability}& \frac{q}{2^{r}(q+2^{r}-1)}
\end{array}\right..$$

\end{thm}

Observe that Theorem \ref{main_gen} specializes to Theorem $1.1$ of \cite{Betal09} when $r=1$ and to Theorem \ref{main} of the present paper when $r=2$.

\newpage

\section*{Appendix}
\begin{center}by Alina Bucur.
\end{center}
\bigskip\noindent {$1.$ \bf Biquadratic covers.}
Fix a finite field $\F_q$ of characteristic different from $2.$ 
A biquadratic cover of $\PP^1$  over $\F_q$ is a covering map $\pi:C \to \PP^1$ such that $\Aut(C/\PP^1) \simeq \Z/2\Z \times \Z/2\Z.$ Such a cover has an affine model (here we take $\PP^1$ to be marked) given by equations 

\[\begin{cases}
y_1^2 =h_1(t) & \\
y_2^2 =  h_2(t) &
\end{cases}\]
with $h_1, h_2 \in \F_q[t]$ both square-free. This corresponds to the field extension $k\left(\sqrt{ h_1}, \sqrt{ h_2}\right)$ of the function field $k = \F_q(t)$ of $\PP_1.$

Secretly, it also has an implied equation \begin{equation}\label{eq:secretr=2}w^2 =  h_1(t)  h_2(t),\end{equation} but of course the right hand side is not necessarily square-free anymore and the gcd of $  h_1$ and $h_2$ will appear squared. If we denote  $f_3 = gcd(  h_1,h_2),$ then we can rewrite our three equations as 

 \[\begin{cases}
y_1^2 =f_1(t)  f_3(t) & \\
y_2^2 = f_2(t) f_3(t) &\\
w^2 = f_1(t) f_2(t)&
\end{cases}\]
with $f_1, f_2, f_3$ all square-free and pairwise coprime.

This corresponds to the field extension $K(\sqrt{f_1f_3}, \sqrt{f_2f_3}, \sqrt{f_1f_2}).$ Note that it is better to think of this as a ``tri-quadratic'' extension, as the roles of $f_1f_3, f_2f_3, f_1f_2$ can be permuted. This also shows that each cover will appear exactly $\#S_3 = 6$ times, which is expected as the automorphism group of the Klein group is indeed $S_3.$ 

Let us look at geometric points over a point $\alpha \in \PP^1$, i.e. we treat our field of definition as algebraically closed. 

First,  looking at \begin{multline*}(f_1,f_2, f_3)\bmod (t-\alpha)^2 =\\ \left(f_1(\alpha) + f_1'(\alpha)(t-\alpha), h_2(\alpha) + f_2'(\alpha)(t-\alpha), f_3(\alpha) + f_3'(\alpha)(t-\alpha) \right) \bmod (t-\alpha)^2\end{multline*} we see that we have:

\begin{itemize}
\item  $(q^2-1)^3$ choices as $f_1,f_2,f_3$ have to be nonzero $\bmod  (t-\alpha)^2.$
\item from those we need to exclude the possibility that $(t-\alpha)$ divides any two of $f_1,f_2,f_3$ (as they have to be coprime), i.e. we cannot allow $f_1(\alpha), f_2(\alpha), f_3(\alpha)$ to contain two or three zeros.  Let us examine each situation we need to avoid.
\begin{itemize}
\item[\bf Two zeros:] The situation is completely symmetric in $f_1,f_2,f_3$ so it is enough to count one possibility and multiply by $3.$ If $f_1(\alpha) = f_2(\alpha) = 0,$ but $f_3(\alpha) \neq 0,$ then we must have $f_1'(\alpha) \neq 0$ (so $q-1$ choices), $f_2'(\alpha) \neq 0$ ($q-1$ choices) and no restrictions on $f_3'(\alpha)$ ($q$ choices). Taking into account the $q-1$ possibilities for $f_3(\alpha)$, we have to subtract $3 (q-1)^3 q.$
\item[\bf Three zeros:] In this case we must have $f_1'(\alpha), f_2'(\alpha), f_3'(\alpha)$ all nonzero, so there are $(q-1)^3$ such triples. 
\end{itemize}
\end{itemize}
Thus we start with \[(q^2-1)^3 - 3q(q-1)^3 -  (q-1)^3 = q^2(q-1)^3(q+3)\] triples modulo  $(t-\alpha)^2.$

Geometrically, there are two possibilities for the fiber over $\alpha.$ 

\begin{enumerate}
\item[{\bf I.}] There will be $4$ distinct points when $f_1(\alpha) f_3(\alpha) \neq 0$ and $f_2(\alpha) f_3(\alpha) \neq 0,$ i.e. when $f_1(\alpha), f_2 (\alpha), f_3(\alpha)$ are all nonzero ($(q-1)$ choices each)  and $f_1' (\alpha), f_2' (\alpha), f_3'(\alpha)$ have no restrictions ($q$ choices each). Thus  there are $q^3(q-1)^3$ possibilities.\\

\item[{\bf II.}] There $2$ distinct points otherwise when either $f_1 (\alpha) f_3(\alpha) =0$ or $f_2(\alpha) f_3(\alpha) =0.$ Since no two of the terms can be zero at the same time, this means that exactly one of $f_1(\alpha), f_2(\alpha), f_3(\alpha)$ is zero. If $f_1(\alpha) =0,$ then $f_1'(\alpha) \neq 0.$ As the situation is again completely symmetric in $f_1, f_2, f_3,$ we have $3q^2(q-1)^3$ such triples. 

\end{enumerate}

To count $\F_q$-rational points, we must split each of the two cases above into two further cases.
\begin{enumerate}
\item[{\bf Ia.}] The fiber has $4$ $\F_q$-rational points when $f_1(\alpha)f_3(\alpha)$ and $f_2(\alpha)f_3(\alpha)$ are both quadratic residues (and non-zero, as we are in the first case above). Thus  $f_1(\alpha), f_2(\alpha), f_3(\alpha)$ have to be all three either quadratic residues or quadratic non residues. In either case they are all nonzero, and thus  $f_1' (\alpha), f_2' (\alpha), f_3'(\alpha)$ have no restrictions imposed on them. Therefore we have 

\[ 2 \left(\frac{q-1}{2} \right)^3 q^3 = \frac{q^3(q-1)^3}{4}\] choices, i.e. probability $1/4$ to get this subcase out of case I.

\item[{\bf Ib.}] There are no rational points (but $2$ points of degree $2$) in the fiber when exactly one of $f_1(\alpha), f_2(\alpha), f_3(\alpha)$ is a  quadratic residue or exactly one of $f_1(\alpha), f_2(\alpha), f_3(\alpha)$ is a  quadratic nonresidue. This situation occurs with probability $3/4$ out of case I. (Note that in this case at least one of the two defining equations has no solution, therefore no rational point.)

\item[{\bf IIa.}] The fiber consists of $2$ $\F_q$-rational points when one of the following three cases occur.
\begin{itemize}
\item $f_1(\alpha)f_3(\alpha) = 0$ and $f_2(\alpha)f_3(\alpha)$ is a nonzero quadratic residue. Thus we need to have $f_1(\alpha)=0,$ $f_1'(\alpha) \neq 0$ and we get \[2 \left( \frac{q-1}{2}\right)^2 (q-1) q^2 = \frac{q^2 (q-1)^3}{2} \textrm{ possibilities.}\] 
\item $f_2(\alpha)f_3(\alpha) = 0$ and $f_1(\alpha)f_3(\alpha)$ is a nonzero quadratic residue. As above,  there are $\frac{q^2 (q-1)^3}{2}$ such triples.
\item $f_3(\alpha) =0$ and $f_1(\alpha) f_2(\alpha)$  is a nonzero quadratic residue.  Similarly, there are $\frac{q^2 (q-1)^3}{2}$ such triples.
\end{itemize}
In conclusion, the probability of getting this subcase out of case II is $1/2.$

\item[{\bf IIb.}] The fiber contains no $\F_q$-rational points (it is one double point of degree $2$) in one of the following three cases.

\begin{itemize}
\item $f_1 (\alpha) f_3(\alpha)= 0$ and $f_2(\alpha)f_3(\alpha)$ is a quadratic nonresidue. There are $\frac{q^2 (q-1)^3}{2}$ such triples.
\item $f_1 (\alpha)f_3(\alpha)$ is a quadratic nonresidue and $f_2(\alpha)h(\alpha) = 0$. As before, there are $\frac{q^2 (q-1)^3}{2}$ such triples.
\item $f_3(\alpha)=0$ and $f_1(\alpha) f_2(\alpha)$ is a quadratic nonresidue. There are also $\frac{q^2 (q-1)^3}{2}$ such triples.
\end{itemize}
In conclusion, the probability of getting this subcase out of case II is also $1/2.$ 
\end{enumerate}

The upshot is the following ``prediction''.

\bigskip\noindent {\bf Conjecture $1$.}\, {\em
\[\prob\left( \#C(\F_q) = M; \textrm{$C$ biquadratic cover of $\PP^1$ defined over $\F_q$}\right) \sim \prob (X_1+ \dots+X_{q+1} = M)\] 

where $X_i$'s are i.i.d. random variables taking values

\[ X_i = \begin{cases} 

4 & \textrm{ with probability } \displaystyle \frac{1}{4} \frac{q^3(q-1)^3}{q^2(q-1)^3(q+3)} = \frac{q}{4(q+3)}\\
&\\
2 &  \textrm{ with probability } \displaystyle \frac{1}{2} \frac{3q^2(q-1)^3}{q^2(q-1)^3(q+3)} = \frac{3}{2(q+3)}\\
&\\
0 &  \textrm{ with probability }  \displaystyle \frac{3q}{4(q+3)}+ \frac{3}{2(q+3)}= \frac{3(q+2)}{4(q+3)}

\end{cases}\]}

Note that the expected number of points in a fiber is $1$ and on the whole curve is $q+1.$ 

\bigskip\noindent {$2.$ \bf The general case: $r$-quadratic covers.}
The argument can be generalized to the case of $r$-quadratic covers of $\PP^1$ over $\F_q$, i.e. covers $\pi:C \to \PP^1$ with \[\Aut (C/\PP^1) \simeq  \underbrace{\Z/2\Z \times \dots \times \Z/2\Z}_{r \textrm{ times}}.\] 

These are precisely the probabilities that appear in Theorem \ref{main_gen} of the present paper. Our argument works for any $r \in \Z_{>0}.$ For $r=2$ we will recover the predictions from Section $1.$ and for $r=1$ we will recover the results from \cite{KR09}.

An affine model of an $r$-quadratic cover is given by equations

\[\begin{cases}
y_1^2 = h_1(t) &\\
y_2^2 = h_2(t)&\\
\vdots &\\
y_r^2 = h_r(t),&
\end{cases}\]
with $ h_1, \dots,  h_r \in \F_q[t]$ square-free polynomials. The cover corresponds to the $r$-quadratic field extension $k\left(\sqrt{ h_1}, \sqrt{ h_2}, \dots, \sqrt{ h_r} \right)$ of $k=\F_q(t).$ 

Together with the ``secret'' equations -- the equivalents of \eqref{eq:secretr=2} -- we get in fact $2^r -1$ equations

\begin{equation}\label{eq:alldef}y_J^2 = \prod_{j \in J} h_j(t)\end{equation}
indexed by the nonempty subsets $J \subseteq \{1, \dots, r\}.$ Again we want to take out gcd's as we did in Section \ref{sec:family}. We obtain $f_1, \dots, f_{2^r-1}$ square-free pairwise coprime polynomials that define this extension. (We can do this by choosing any enumeration of the nonempty subsets $J$. One possibility is the one described in the proof of Theorem 6.3 in the paper.) 

Let us examine the fiber above a point $\alpha \in \PP^1.$ We first consider the geometric points. In order to ease notation, let $m = 2^r-1.$ We need to look at 

\[(f_1, \dots, f_m) \equiv \left(f_1(\alpha) + f_1'(\alpha)(t-\alpha), \dots, f_m(\alpha) + f_m'(\alpha) (t-\alpha) \right) \bmod (t-\alpha)^2.\]

First, since \[(f_1, \dots, f_m) \not \equiv (0,\dots,0) \bmod (t-\alpha)^2,\] we start with at most $(q^2-1)^m$ choices. From these, we need to exclude those that would allow two or more of the $f_i$'s to be divisible by $(t-\alpha),$ since they have to be pairwise coprime. 

Fix an integer $k$ with $1 \leq k \leq m.$ Note that if have exactly $k$ zeros among $f_1(\alpha), \dots, f_m(\alpha),$ then the corresponding derivatives of the $k$ polynomials that have a zero at $\alpha$ must be nonzero. Thus, for each choice of $k$ numbers in $\{1, \dots, m\}$ we have \[(q-1)^k (q-1)^{m-k} q^{m-k} =q^{m-k} (q-1)^m \textrm{ choices}.\] Therefore we start with 

\[(q^2-1)^m - \sum_{k=2}^m \binom{m}{k}  q^{m-k} (q-1)^m %= (q^2-1)^m - (q-1)^m \left[ - q^m - m q^{r-1}+ \sum{k=0}^m \binom{m}{k}  q^{m-k}\right] = 
=q^{m-1}(q-1)^m (q+m) \textrm{ tuples modulo }(t-\alpha)^2.\] 

Geometrically, we have the following possibilities. 

\begin{enumerate}
\item[\bf I.] There will be $2^r$ distinct points when $f_1(\alpha), \dots,  f_m(\alpha)$ are all nonzero. Again in this case we have no restrictions on the derivatives. Thus we have $(q-1)^{m}q^m$ possibilities.\\

\item[\bf II.] Otherwise, only one of the $f_1(\alpha), \dots,  f_m(\alpha)$ can be zero (as the $f_j$s are pairwise coprime). In this case, the fiber will contain $2^{r-1}$ geometric points. The situation is completely symmetric in $f_1, \dots, f_m.$ Thus, it is enough to count the case when $f_1(\alpha) =0$ and multiply the result by $m.$  Then $f_1'(\alpha) \neq 0$ so the derivative can take $(q-1)$ values. For the other $m-1$ terms, we have $f_j(\alpha) \neq 0$, so it can take $(q-1)$ values, and there are no restrictions on the derivatives $f'_j(\alpha), 2 \leq j \leq m.$  This means that there are $m (q-1) (q-1)^{m-1}q^m$ choices in total that lead to this case. Note that this is equal to \[mq^{m-1}(q-1)^m= q^{m-1}(q-1)^m (q+m) - q^{m}(q-1)^{m}.\]

\end{enumerate}

We now look at the $\F_q$-rational points in the fiber above $\alpha.$ The two cases above split into two cases each.

\begin{enumerate}
\item[{\bf Ia.}] The fiber has $2^r$ $\F_q$-rational points which occurs with probability $1/2^r$ out of case I.

\item[{\bf Ib.}] The fiber has no $\F_q$-rational points which occurs with probability $(2^r-1)/2^r$ out of case I.

\item[{\bf IIa.}] The fiber consists of $2^{r-1}$ $\F_q$-rational points which occurs with probability $1/{2^{r-1}}$ out of case II.

\item[{\bf IIb.}] The fiber contains no $\F_q$-rational points which occurs with probability $(2^{r-1}-1)/{2^{r-1}}$ out of case II.
\end{enumerate}

Note that this already tells us that the expected number of points in a fiber is $1$ and on the whole curve is $q+1.$ Another interesting observation is that the case of biquadratic extensions is a bit different than the general case. For instance, for $r=2$ case II splits into two subcases with probability 50-50; but in general the probabilities for the subcases IIa and IIb are $2^{1-r}$ and $(1 - 2^{1-r}).$ Which means that one has to look at the case $r\geq 3$ in order to get the complete picture.

Since $m = 2^r-1$ we get the following prediction. 

\bigskip\noindent {\bf Conjecture $2$.}\, {\em

\[\prob\left( \#C(\F_q) = M; \textrm{$C$ $r$-quadratic cover of $\PP^1$ defined over $\F_q$}\right) \sim \prob (X_1+ \dots+X_{q+1} = M)\] 

where $X_i$'s are i.i.d. random variables taking values

\[ X_i = \begin{cases} 

2^r & \textrm{ with probability } \displaystyle \frac{1}{2^r} \cdot \frac{q^m(q-1)^m}{q^{m-1}(q-1)^m(q+m)} = \frac{q}{2^r(q+m)}= \frac{q}{2^r(q+2^r -1)}\\
&\\
2^{r-1} &  \textrm{ with probability } \displaystyle \frac{1}{2^{r-1}} \cdot \frac{mq^{m-1}(q-1)^m}{q^{m-1}(q-1)^m(q+m)} = \frac{m}{2^{r-1}(q+m)} = \frac{2^r-1}{2^{r-1}(q+2^r -1)}\\
&\\
0 &  \textrm{ with probability }  \displaystyle \frac{(2^r-1)(q+2^r-2)}{2^r(q+2^r-1)}.

\end{cases}\]}

These are exactly the probabilities that appear in Theorem \ref{main_gen} of the present paper. 

The argument works for any $r \in \Z_{>0}.$ For $r=2$ we recover the predictions from previous Section; for $r=1$ we recover the random variables from \cite{KR09}.

%\cleardoublepage
%\appendix
%\section{Heuristics for biquadratic and $r$-quadratic covers over finite fields}
%\author{Alina Bucur}
%\input{biquadrheuristics}

\end{document}